\documentclass[a4paper]{amsart}

\usepackage[T2A]{fontenc}
\usepackage[utf8]{inputenc}
\usepackage[english]{babel}

\usepackage{amsmath, amssymb, amsfonts, amsthm}
\usepackage{mathtools}

\usepackage{hyperref}

\usepackage{microtype}

\allowdisplaybreaks
\usepackage{enumitem}

\newcommand{\sign}{\operatorname{sign}}
\newcommand{\T}{\operatorname{T}}
\renewcommand{\P}{\operatorname{P}}

\theoremstyle{plain}
\newtheorem{theorem}{Theorem}
\newtheorem{lemma}{Lemma}
\newtheorem{definition}{Definition}

\newtheorem{conjecture}{Conjecture}

\title{Explicit formula for the Rasmussen invariant of odd pretzel knots}

\author{Yury Belousov}
\address{Yury Belousov \\ Leonhard Euler International Mathematical Institute in Saint Petersburg}
\email{bus99@yandex.ru}

\author{Vadim Stepaniuk}
\address{Vadim Stepaniuk \\ Saint Petersburg State University}
\email{stepaniuk.vadim.math@gmail.com}

\date{}

\begin{document}

    \begin{abstract}
        We compute the Rasmussen $s$-invariant for all odd pretzel knots, giving an explicit closed formula. The formula extends to all slice-torus invariants that are invariant under mutation. In the absence of mutation invariance, we prove the formula under a mild condition on the pretzel parameters, which we conjecture to be unnecessary.
    \end{abstract}

    \maketitle


    \section*{Introduction}

        Rasmussen introduced the celebrated $s$-invariant in~\cite{Ras10}, extracting it from Lee's deformation of Khovanov homology~\cite{Lee05}. It has since proved to be a powerful tool in smooth knot theory: Rasmussen used it to give the first purely combinatorial proof of the Milnor conjecture, and Piccirillo later used it to show that the Conway knot is not smoothly slice~\cite{Pic20}.

        One natural way to understand the behaviour of a knot invariant is to compute it on families of knots with simple diagrammatic descriptions. Among these, pretzel knots stand out as a versatile testing ground. Recall that a \emph{pretzel link}~$\P(a_{1}, \ldots, a_{n})$ is a link admitting a diagram consisting of $n$ blocks, where the $i$-th block contains $|a_{i}|$ crossings twisted according to the sign of $a_{i}$, and the blocks are connected cyclically, as shown in Figure~\ref{fig:pretzel_ex}; such a diagram is called \emph{standard}.

        A pretzel link has exactly one component (i.e. it is a knot) if and only if either precisely one $a_{i}$ is even, or all $a_{i}$ are odd and $n$ is odd. In the present paper, we consider knots satisfying the latter condition, which are called \emph{odd} pretzel knots. For every pretzel knot, we orient the standard diagram so that the bottom arc is directed from left to right, as shown on the left in Figure~\ref{fig:pretzel_ex}.

        However, despite the simplicity of their diagrams, the $s$-invariant of pretzel knots has so far been computed only in special cases. Suzuki~\cite{Suz10} settled the three-strand case almost completely, and Kim and Sano~\cite{KS25} later completed it, also giving a new diagrammatic proof of the known cases. For pretzel knots with more blocks, only certain families of even pretzel knots had been treated, in the work of Kim and Yeon~\cite{KY15} and Kawamura~\cite{Kaw15}. In the present paper, we give an explicit formula valid for every odd pretzel knot, as follows.

        \begin{theorem}\label{thm:s_value_odd}
    		Let $n>0$ be odd and $a_{1}, \ldots, a_{n}$ be odd integers. Then
            \begin{equation}\label{eq:s_value_odd}
                s\bigl(\P(a_{1}, \ldots, a_{n})\bigr) = -\sum_{i=1}^{n} \sign(a_{i}) + \sign(a_{\min}),
            \end{equation}
            where $a_{\min}$ is the parameter with minimal absolute value in the multiset obtained from the multiset $\{a_{1}, \ldots, a_{n}\}$ by iteratively removing pairs of the form $\{t, -t\}$ until no such pair remains.
        \end{theorem}

        \begin{figure}[h]
            \centering
            \includegraphics[width=0.61\linewidth]{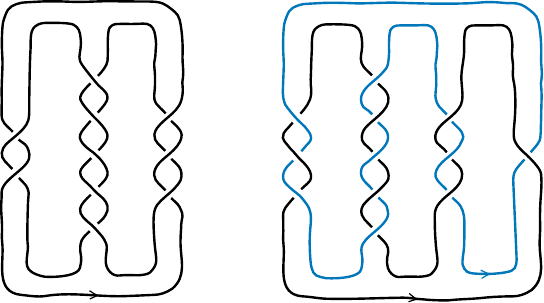}
            \caption{Examples of pretzel links: the knot $\P(2, -5, 3)$ and link $\P(3, 5, -3, 1)$.}
            \label{fig:pretzel_ex}
        \end{figure}

        Our proof of Theorem~\ref{thm:s_value_odd} does not rely on any specifics of the $s$-invariant beyond the following two facts:
        \begin{enumerate}[leftmargin=*, topsep=-2pt, itemsep=0pt]
            \item $\frac{1}{2}s$ is a slice-torus invariant (see Definition~\ref{def:slice_torus} and~\cite[Theorems 1, 2 and 4]{Ras10});
            \item $s$ is invariant under mutation (see~\cite[Theorem 1.16]{KWZ19}).
        \end{enumerate}
        Consequently, the same argument applies to any mutation-invariant slice-torus invariant.

        The notion of a slice-torus invariant goes back to Livingston~\cite{Liv04}, who showed that the Ozsv\'{a}th--Szab\'{o} $\tau$-invariant, defined in~\cite{OS03}, satisfies the properties later used to define this class (see Definition~\ref{def:slice_torus}). The term ``slice-torus invariant'' itself was introduced by Lewark~\cite{Lew14}, who initiated a systematic study of such invariants. Besides the aforementioned $\frac{1}{2}s$ and $\tau$, this class contains many invariants arising from various link homology theories, including the normalised $s_{N}$-invariants from Khovanov--Rozansky homology~\cite{Wu09, Lob09, Lob12}, the $q_{M}$-invariant from $\mathbb{Z}_{2}$-equivariant Seiberg--Witten Floer homology~\cite{IT24}.

        For most slice-torus invariants, mutation invariance has not been established. To handle this case, we replace the mutation assumption by a mild condition on the parameter of minimal absolute value, under which the same formula still holds.

        \begin{theorem}\label{thm:value_odd}
            Let $\nu$ be a slice-torus invariant, $n>0$ be odd, and $a_{1}, \ldots, a_{n}$ be odd integers. Let $a_{\min}$ be the parameter with minimal absolute value in the multiset obtained from the multiset $\{a_{1}, \ldots, a_{n}\}$ by iteratively removing pairs of the form $\{t, -t\}$ until no such pair remains.
            Then
            \begin{equation}\label{eq:value_odd}
                2\nu\bigl(\P(a_{1}, \ldots, a_{n})\bigr) = -\sum_{i=1}^{n} \sign(a_{i}) + \sign(a_{\min})
            \end{equation}
            provided that at least one of the following conditions holds:
            \begin{enumerate}[leftmargin=*, topsep=-2pt, itemsep=0pt, label=\textup{(\alph*)}]
                \item\label{value_odd_case_a} $\nu$ is mutation invariant;
                \item\label{value_odd_case_b} all parameters $a_i$ satisfying $|a_{i}|=\min\limits_{1\le j\le n}|a_{j}|$ have the same sign.
            \end{enumerate}
        \end{theorem}

        \vspace*{-0.45\baselineskip}

        We have checked formula~\eqref{eq:value_odd} for the $\tau$-invariant and the $s_{3}$-invariant (neither of which is known to satisfy condition~\ref{value_odd_case_a}) on a large number of odd pretzel knots not covered by case~\ref{value_odd_case_b}, using simple computer scripts available in the repository~\cite{Ste26}. Since all computed values agree with~\eqref{eq:value_odd}, we propose the following conjecture.

        \begin{conjecture}\label{conj:value_odd}
    		Let $\nu$ be a slice-torus invariant, $n>0$ be odd, and $a_{1}, \ldots, a_{n}$ be odd integers. Then formula~\eqref{eq:value_odd} holds.
        \end{conjecture}


    \subsection*{Acknowledgements}
        We are grateful to Dirk Sch\"{u}tz for making KnotJob~\cite{KJ} available, which allowed us to identify patterns in the values of the $s$-invariant on small pretzel knots. We also thank Viktoriia Georgievskaia for many helpful discussions.


    \section{Preliminaries}

        In this section we record the definition of slice-torus invariants and the few important properties that will be used throughout. All knots are assumed to be oriented and smooth.

        \begin{definition}\label{def:slice_torus}
            A real-valued knot invariant $\nu$ is called a \emph{slice-torus invariant} if it satisfies the following properties:
            \begin{enumerate}[leftmargin=!, topsep=-2pt, itemsep=0pt, align=left, font=\normalfont]
                \item[additivity.] $\nu(K_{1} \# K_{2}) = \nu(K_{1}) + \nu(K_{2})$ for every pair of knots $K_{1}, K_{2}$;
                \item[slice.] $|\nu(K)| \leqslant g_{4}(K)$ for every knot $K$;
                \item[torus.] $\nu(\T(p, q)) = g_{4}(\T(p, q))$ for all coprime positive integers $p, q$.
            \end{enumerate}
        \end{definition}

        Here $g_{4}(K)$ denotes the smooth four-genus of $K$, namely the minimum genus of a compact, connected, oriented surface smoothly and properly embedded in $\mathbb{D}^{4}$ with boundary $K \subset \mathbb{S}^{3} = \partial \mathbb{D}^{4}$. We write $g_{3}(K)$ for the Seifert genus. The notation $-K$ denotes the concordance inverse of $K$, namely its mirror with reversed orientation.

        Soon after proving that the $\tau$-invariant satisfies the properties of Definition~\ref{def:slice_torus}, Livingston investigated its behaviour under crossing changes, yielding the following lemma.

        \begin{lemma}[{\cite[Corollary 3]{Liv04}}]\label{lem:change_cross}
            Let $\nu$ be a slice-torus invariant, and let knots $K_{+}$ and~$K_{-}$ differ by a single crossing change from positive to negative. Then
            $$
            0 \leqslant \nu(K_{+}) - \nu(K_{-}) \leqslant 1.
            $$
        \end{lemma}

        A crossing change is a special case of adding a single full twist on two strands, as shown in Figure~\ref{fig:change_crossing}. The following lemma, called the full twist lemma, describes the behaviour of slice-torus invariants under the local addition of an arbitrary number of full twists.

        \begin{figure}[h]
            \centering
            \includegraphics[width=0.61\linewidth]{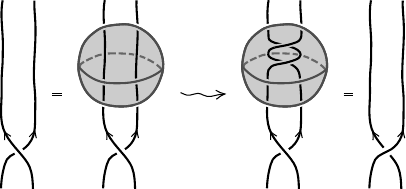}
            \caption{A crossing change realised by adding a single full twist of strands oriented in the same direction.}
            \label{fig:change_crossing}
        \end{figure}

        \begin{lemma}[{\cite[Propositions 2.1 and 2.4]{LZ24}}]\label{lem:full_twist}
            Let $K$ be a knot, and let $B$ be a $3$-ball intersecting $K$ in a trivial two-stranded tangle as in Figure~\ref{fig:full_twist_lemma}. For an integer $n$, let $K_{n}$ be obtained from $K$ by replacing this tangle inside $B$ with one having $n$ full twists. Let~$\nu$ be a slice-torus invariant, and let $m \geqslant n$ be integers. If the two strands of the tangle are oriented in opposite directions, then
            $$
            -1 \leqslant \nu(K_{m}) - \nu(K_{n}) \leqslant 0,
            $$
            and if they are oriented in the same direction, then
            $$
            m-n-1 \leqslant \nu(K_{m}) - \nu(K_{n}) \leqslant m-n.
            $$
        \end{lemma}

        Although~\cite{LZ24} states these propositions for integer-valued slice-torus invariants, the inequalities stated here follow from the cobordism and crossing-change arguments in their proofs without using integrality.

        \begin{figure}[h]
            \centering
            \includegraphics[width=0.61\linewidth]{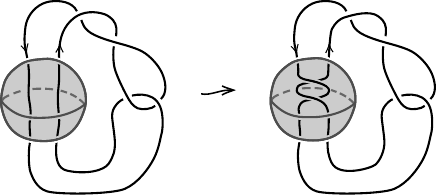}
            \caption{
            Inserting one negative full twist inside the grey 3-ball with strands oriented in opposite directions: $4_1$ becomes $6_1$.}
            \label{fig:full_twist_lemma}
        \end{figure}

        Recall that the \emph{Seifert graph} $\Gamma(D)$ of a diagram $D$ is the graph whose vertices correspond to the Seifert circles of $D$ and whose edges correspond to the crossing bands joining these circles in the standard Seifert surface.

        The following result of Lewark, known as the slice-Bennequin inequality, provides upper and lower bounds for slice-torus invariants directly from a diagram.

        \begin{lemma}[{\cite[Theorem 5]{Lew14}}]\label{lem:bennequin}
            Let $D$ be a diagram of a knot $K$ with writhe $w(D)$ and Seifert circle number $c(D)$. Let $O_{\pm}(D)$ denote the number of connected components of the graph obtained from the Seifert graph $\Gamma(D)$ of $D$ by deleting all edges corresponding to crossings of sign $\mp$. Then for any slice-torus invariant $\nu$, we have
            \begin{equation}\label{ineq:benneq}
                -1 + w(D) - c(D) + 2O_{+}(D) \leqslant 2\nu(K) \leqslant 1 + w(D) + c(D) - 2O_{-}(D).
            \end{equation}
        \end{lemma}


    \section{Proofs of the main theorems}

        To begin, we explicitly write down the bounds in inequality~\eqref{ineq:benneq} for odd pretzel knots.

        \begin{lemma}\label{lem:benneq_odd}
    		Let $\nu$ be a slice-torus invariant, $n>0$ be odd, and $a_{1}, \ldots, a_{n}$ be odd integers. Then
    		$$
    		-\sum_{i=1}^{n} \sign(a_{i}) - 1 \leqslant 2\nu\bigl(\P(a_{1}, \ldots, a_{n})\bigr) \leqslant -\sum_{i=1}^{n} \sign(a_{i}) + 1.
    		$$
    	\end{lemma}
        \begin{proof}
            Apply the Bennequin-type inequality~\eqref{ineq:benneq} to the standard oriented diagram $D$ of the odd pretzel knot $\P(a_{1}, \ldots, a_{n})$. The Seifert graph of $D$ consists of two distinguished vertices joined by $n$ paths. The $i$-th path corresponds to the $i$-th twist block and has~$|a_{i}|$ edges, all of sign $-\sign(a_{i})$, and $|a_{i}|-1$ internal vertices. See Figure~\ref{fig:odd_seifert_ex} for an example.

            \begin{figure}[h]
                \centering
                \includegraphics[width=0.58\linewidth]{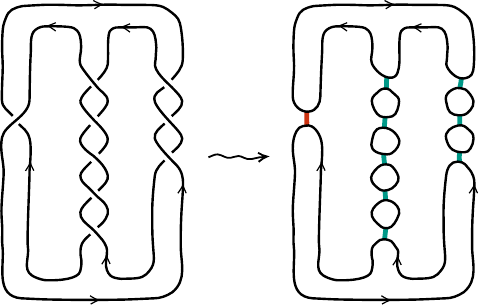}
                \caption{The standard diagram of the pretzel knot $\P(1, -5, -3)$ and its Seifert graph. Red edges correspond to negative crossings, green edges to positive ones.}
                \label{fig:odd_seifert_ex}
            \end{figure}

            Represent the multiset $\{a_{1}, \ldots, a_{n}\}$ as the union of positive parameters $p_{1}, \ldots, p_{c_{+}} > 0$ and negative ones $-q_{1}, \ldots, -q_{c_{-}} < 0$. For these two subsets, set
            $$
            P = \sum_{i=1}^{c_{+}} p_{i} \qquad \text{and} \qquad Q = \sum_{j=1}^{c_{-}} q_{j}.
            $$

            Since the sign of every crossing in a block is opposite to the sign of the corresponding parameter, the writhe of $D$ is
            $$
            w(D) = -\sum_{i=1}^{n} a_{i} = Q - P,
            $$
            and the number of Seifert circles is
            $$
            c(D) = 2 + \sum_{i=1}^{n} (|a_{i}|-1) = 2 + P-c_{+} + Q-c_{-}.
            $$

            After the negative edges are deleted from the Seifert graph, the $p_{i}-1$ internal vertices of each path corresponding to a positive parameter $p_{i}$ become isolated. The paths corresponding to negative parameters remain intact. If $c_{-} \geqslant 1$, these paths and the two distinguished vertices form a single connected component. If $c_{-} = 0$, the two distinguished vertices form two separate components. Hence the number of connected components of the Seifert graph after deleting the negative edges equals
            $$
            O_{+}(D) =
            \begin{cases}
                \sum_{i=1}^{c_{+}} (p_{i}-1) + 1 = P-c_{+} + 1 & \text{if } c_{-} \geqslant 1,\\
                \sum_{i=1}^{c_{+}} (p_{i}-1) + 2 = P-c_{+} + 2 & \text{if } c_{-} = 0.
            \end{cases}
            $$

            Similarly, the number of connected components after deleting the positive edges is
            $$
            O_{-}(D) =
            \begin{cases}
                \sum_{j=1}^{c_{-}} (q_{j}-1) + 1 = Q-c_{-} + 1 & \text{if } c_{+} \geqslant 1,\\
                \sum_{j=1}^{c_{-}} (q_{j}-1) + 2 = Q-c_{-} + 2 & \text{if } c_{+} = 0.
            \end{cases}
            $$

            Assume first that $c_{\pm} \geqslant 1$. Applying the Bennequin inequality~\eqref{ineq:benneq}, we get
            \begin{align*}
                2\nu(K) & \leqslant 1 + w(D) + c(D) - 2O_{-}(D) = \\
                        & = 1 + (Q-P) + \bigl(2 + P-c_{+} + Q-c_{-}\bigr) - 2\bigl(Q - c_{-} + 1\bigr) = \\
                        & = c_{-} - c_{+} + 1,
            \end{align*}
            and similarly,
            \begin{align*}
                2\nu(K) & \geqslant -1 + w(D) - c(D) + 2O_{+}(D) = \\
                        & = -1 + (Q-P) - \bigl(2 + P-c_{+} + Q-c_{-}\bigr) + 2\bigl(P - c_{+} + 1\bigr) = \\
                        & = c_{-} - c_{+} - 1,
            \end{align*}
            which, together with the identity $-\sum_{i=1}^{n} \sign(a_{i}) = c_{-} - c_{+}$, completes the proof in this case.

            If $c_{-} = 0$, the upper bound $2\nu(K) \leqslant c_{-} - c_{+} + 1$ still holds, while the lower bound becomes
            \begin{align*}
                2\nu(K) & \geqslant -1 + (Q-P) - \bigl(2 + P-c_{+} + Q-c_{-}\bigr) + 2\bigl(P - c_{+} + 2\bigr) = \\
                        & = c_{-} - c_{+} + 1,
            \end{align*}
            so $2\nu(K) = c_{-} - c_{+} + 1$.

            The case $c_{+} = 0$ is analogous and gives $2\nu(K) = c_{-} - c_{+} - 1$.
        \end{proof}

        We proceed to the proof of Theorem~\ref{thm:value_odd}. We prove case~\ref{value_odd_case_b} first, and then deduce case~\ref{value_odd_case_a} from it.


        \begin{proof}[Proof of Theorem~\ref{thm:value_odd}, case~\ref{value_odd_case_b}]
            We use the notation from the proof of Lemma~\ref{lem:benneq_odd}. Consider the pretzel knot $K = \P(a_{1}, \ldots, a_{n})$.

            The case $n = 1$ is trivial, since $K$ is the unknot. If $c_{+} = 0$ or $c_{-} = 0$, then the upper and lower bounds in the Bennequin inequality coincide, as is immediate from the proof of Lemma~\ref{lem:benneq_odd}, and are equal to the right-hand side of~\eqref{eq:value_odd}.

            Now assume that $c_{\pm} \geqslant 1$. We may suppose $a_{\min} > 0$: otherwise replace $K$ by its concordance inverse $-K = \P(-a_{n}, \dots, -a_{1})$. Since $\nu(-K) = -\nu(K)$, the formula for $K$ follows from that for $-K$. Finally, the pretzel knot is invariant under cyclic permutation of the blocks, so we may further assume $a_{1} = a_{\min}$.

            For a fixed $n$, we argue by induction on $c_{+}$, the number of positive parameters.

            Consider the base case $c_{+} = 1$. Then
            $$
            K = \P(p_{1}, -q_{1}, \ldots, -q_{c_{-}}),
            $$
            where $p_{1} = a_{1}$. The assumptions of the theorem give $p_{1}-q_{i} < 0$, and $-q_{i}-q_{j} < 0$ for all $1 \leqslant i, j \leqslant c_{-}$. By a theorem of Rudolph~\cite[Theorem 4.5]{Rud01}, such pretzel knots are strongly quasipositive. A result of Livingston~\cite[Theorem 4]{Liv04} then implies that every slice-torus invariant of $K$ agrees with its three-genus, which for odd pretzel knots was computed by Gabai~\cite[Theorem 3.2]{Gab86} to equal $\frac{n-1}{2}$. Consequently,
            $$
            2\nu(K) = 2g_{3}(K) = n-1 = c_{-} - c_{+} + 1.
            $$

            Now let $c_{+} \geqslant 2$ and choose $a_{j} > 0$ with $j \neq 1$. Consider the knots
            \begin{gather*}
                K = \P(a_{1}, \ldots, a_{j}, \ldots, a_{n}), \\
                K' = \P(a_{1}, \ldots, -a_{j}-2, \ldots, a_{n}).
            \end{gather*}
            They are related by the operation of adding full twists on two strands in the $j$-th block, performed inside a $3$-ball as shown in Figure~\ref{fig:ft_odd_ex}. In the standard diagram of an odd pretzel knot, the strands in each twist block are oppositely oriented, so adding full twists falls under the first case of Lemma~\ref{lem:full_twist}. The knot $K$ is obtained from $K'$ by adding~$a_{j} + 1 > 0$ full twists, hence
            $$
            -2 \leqslant 2\nu(K) - 2\nu(K') \leqslant 0.
            $$

            \begin{figure}[h]
                \centering
                \includegraphics[width=0.6\linewidth]{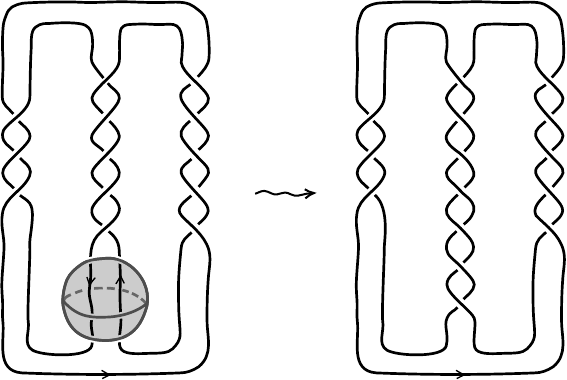}
                \caption{Left: the knot $K = \P(3, 5, -5)$ with a grey ball $B$. Right: the knot $K' = \P(3, -7, -5)$, obtained from $K$ by adding $-6$ full twists inside $B$.}
                \label{fig:ft_odd_ex}
            \end{figure}

            Since $\lvert -a_{j}-2 \rvert > |a_{j}| \geqslant |a_{\min}|$ and all other parameters remain unchanged, the parameter of minimal absolute value in $K'$ is still $a_{\min}$, and it appears with only one sign. Thus $K'$ satisfies the hypotheses of Theorem~\ref{thm:value_odd}. This knot has $c_{+}-1$ positive and $c_{-}+1$ negative blocks, so by the induction hypothesis,
            $$
            2\nu(K') = (c_{-}+1) - (c_{+}-1) + 1 = c_{-} - c_{+} + 3,
            $$
            whence
            \begin{gather*}
                -2 \leqslant 2\nu(K) - (c_{-} - c_{+} + 3) \leqslant 0, \\
                c_{-} - c_{+} + 1 \leqslant 2\nu(K) \leqslant c_{-} - c_{+} + 3.
            \end{gather*}

            On the other hand, Lemma~\ref{lem:benneq_odd} yields $2\nu(K) \leqslant c_{-} - c_{+} + 1$, so together with the lower bound just obtained, this forces equality.
        \end{proof}


        \begin{proof}[Proof of Theorem~\ref{thm:value_odd}, case~\ref{value_odd_case_a}]
            In this case, $\nu$ is invariant under mutation. Hence its value on a pretzel knot is independent of the order of the blocks: any transposition of two adjacent blocks can be realised by a mutation, as illustrated in Figure~\ref{fig:mutation}.

            \begin{figure}[h]
                \centering
                \includegraphics[width=0.6\linewidth]{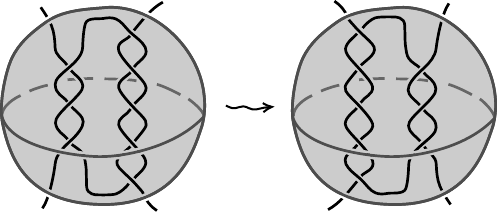}
                \caption{Realising a transposition of blocks with parameters $3$ and $-5$ by a mutation.}
                \label{fig:mutation}
            \end{figure}

            Now suppose the tuple of parameters contains two adjacent parameters of the form $t$ and $-t$. Then the knot can be replaced by a concordant pretzel knot with this pair removed, as shown in Figure~\ref{fig:remove_adjacent}. Since slice-torus invariants are concordance invariants, the value of $\nu$ does not change under such a reduction.

            \begin{figure}[h]
                \centering
                \includegraphics[width=0.8\linewidth]{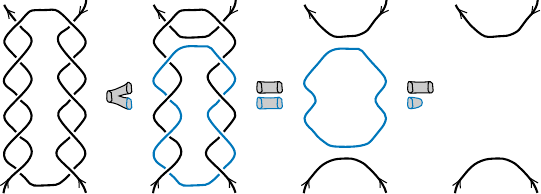}
                \caption{The operation of cancelling an adjacent pair $5, -5$ by a genus-zero cobordism.}
                \label{fig:remove_adjacent}
            \end{figure}

            By repeatedly applying these two operations --- permuting blocks to make a pair $t, -t$ adjacent, and then cancelling it --- we obtain a pretzel knot $K'$ whose parameter multiset has no such pairs and such that $\nu(K) = \nu(K')$. In particular, the parameter of minimal absolute value in $K'$ appears with only one sign, so $K'$ satisfies the hypotheses of case~\ref{value_odd_case_b}, which has already been proved. Moreover, since each cancelled pair consists of one positive and one negative parameter, the sum $\sum_{i} \sign(a_{i})$ is the same for $K$ and~$K'$. Applying the already established formula to $K'$ and using the equality of the sign sums, we obtain the desired formula for $K$.
        \end{proof}


    \section{Further remarks}

        \subsection{Applications}

            Formula~\eqref{eq:s_value_odd} for the $s$-invariant follows immediately from the fact that~$\frac{1}{2}s$ is a mutation-invariant slice-torus invariant, as noted in the introduction. We note that the same formula could also be proved without invoking Rudolph's quasipositivity criterion: for the base case $c_{+} = 1$, one may use the diagrammatic method of Kim and Sano~\cite[Proposition 5.9]{KS25}.

            The computation of the Rasmussen invariant for odd pretzel knots was originally motivated by the question of smooth sliceness for the topologically slice knots constructed by the first author in~\cite{Bel26}. The knot $\P(13, -15, -17, 29, 71)$ is one such example. Those knots all have two negative and three positive parameters, all distinct in absolute value, with the parameter of minimal absolute value being positive. Applying Theorem~\ref{thm:s_value_odd} gives $s = 2 - 3 + 1 = 0$ for each of them. Unfortunately, this yields no information about their smooth sliceness. Nonetheless, the investigation of these examples led to the general formula.

            Combining the formula~\eqref{eq:s_value_odd} with Bryant's reformulation~\cite[Theorem 2.1]{Bry17} of Jabuka's formula for the signature of odd pretzel knots~\cite[Theorem~1.18(3)]{Jab10}, we obtain the equality
            \begin{equation}\label{eq:s_plus_signature}
                (s + \sigma)\bigl(\P(a_{1}, \ldots, a_{n})\bigr) = \sign(a_{\min}) - \sign\left(\sum_{i=1}^{n}\frac{1}{a_{i}}\right),
            \end{equation}
            which allows one to determine directly from the multiset of parameters whether the invariants~$s$ and $-\sigma$ agree on the given knot.

            The explicit formulae also allow us to construct infinite families of knots on which $s$ and~$-\sigma$ agree or differ. Indeed, multiplying all parameters by the same positive odd integer leaves both invariants unchanged. Hence, starting from the knot $\P(-3, 5, 7, 9, -11)$, for which $s = -2$ and $-\sigma = 0$, we obtain a family $\bigl\{\P(-3k, 5k, 7k, 9k, -11k)\bigr\}$ with the same values of $s$ and $-\sigma$, where $k$ ranges over the positive odd integers. These knots are pairwise distinct by Kawauchi's classification of pretzel knots~\cite[Theorem I]{Kaw85}.


        \subsection{Open questions}
            As far as we know, the normalised Rasmussen invariants $\frac{1}{2} s_{\mathbb{F}}$, where $\mathbb{F}$ ranges over fields, are currently the only slice-torus invariants for which mutation invariance has been proved, and no slice-torus invariant is known to fail mutation invariance. Consequently, case~\ref{value_odd_case_a} of Theorem~\ref{thm:value_odd} applies, at present, only to these invariants.

            A representative example of an odd pretzel knot that does not satisfy the hypotheses of case~\ref{value_odd_case_b} of Theorem~\ref{thm:value_odd} is the knot $\P(7, 3, 5, -3, -5)$. In the work~\cite{FLL25}, the authors introduced the notion of squeezed knots, on which all slice-torus invariants agree. Direct attempts to construct a squeezing cobordism for this knot have been unsuccessful.

            The knot $K = \P(7, 3, 5, -3, -5)$ is mutant to the ribbon knot $\P(7, 3, -3, 5, -5)$, yet is not topologically slice by a remark of Herald, Kirk and Livingston~\cite{HKL10}. Quasipositive knots satisfy $s = 2g_{4}$ by a result of Shumakovitch~\cite[Proposition~1.G]{Shu07}. Since~$g_{4}(K) > 0$ and $s(K) = 0$ by Theorem~\ref{thm:s_value_odd}, the knot $K$ is not quasipositive. The same argument applied to the mirror of $K$ shows that $K$ is not quasinegative either. Furthermore,~$K$ is neither homogeneous by the result of Stoimenow~\cite[Theorem~4.1]{Sto08} nor quasialternating by the result of Greene~\cite[Theorem 1.4]{Gre10}. It is not covered by any of the families from Proposition~1.2 of~\cite{FLL25}, with the possible exception of the family of pseudoalternating knots. Thus this pretzel knot appears to lie beyond the reach of the standard tools for computing slice-torus invariants.

    
    \bibliographystyle{alpha}
    \bibliography{references}
\end{document}